\documentclass[a4paper,12pt]{amsart}
\title{The true complexity of a system of linear equations}
\author{W.T. Gowers}
\address{Department of Pure Mathematics and Mathematical Statistics, Wilberforce Road, Cambridge CB3 0WB, UK.}
\email{W.T.Gowers@dpmms.cam.ac.uk}
\author{J. Wolf}
\address{Department of Pure Mathematics and Mathematical Statistics, Wilberforce Road, Cambridge CB3 0WB, UK.}
\email{J.Wolf@dpmms.cam.ac.uk}
\date{1st November 2007}

\usepackage{amsmath, wrapfig}
\usepackage{dsfont, a4wide, amsthm, amssymb, amsfonts, graphicx}
\usepackage{fancyhdr, xspace, psfrag, setspace, supertabular, color}

\newtheorem{theorem}{Theorem}[section]

\newtheorem{lemma}[theorem]{Lemma}

\newtheorem{corollary}[theorem]{Corollary}
\newtheorem{conjecture}[theorem]{Conjecture}
\newtheorem{definition}[theorem]{Definition}

\def\eps{\epsilon}
\def\E{\mathbb{E}}
\def\Z{\mathbb{Z}}
\def\R{\mathbb{R}}

\def\C{\mathbb{C}}
\def\N{\mathbb{N}}

\def\F{\mathbb{F}}

\def\a{\mathbf{a}}
\def\b{\mathbf{b}}

\def\n{\mathbf{n}}

\def\w{\mathbf{w}}
\def\x{\mathbf{x}}
\def\y{\mathbf{y}}

\def\oomega{\boldsymbol{\omega}}

\def\Lsys{\mathcal{L}}
\def\bone{\mathcal{B}_1}
\def\btwo{\mathcal{B}_2}
\def\B{\mathcal{B}}

\def\tends{\rightarrow}

\begin{document}

\maketitle

\begin{abstract} In this paper we look for conditions that are sufficient
to guarantee that a subset $A$ of a finite Abelian group $G$ contains the
``expected'' number of linear configurations of a given type. The simplest 
non-trivial result of this kind is the well-known fact that
if $G$ has odd order, $A$ has density $\alpha$ and all Fourier coefficients 
of the characteristic function of $A$ are significantly smaller than 
$\alpha$ (except the one at zero, which equals $\alpha$), then $A$ contains 
approximately $\alpha^3|G|^2$ triples of the form $(a,a+d,a+2d)$. This is
``expected'' in the sense that a random set $A$ of density $\alpha$ has
approximately $\alpha^3|G|^2$ such triples with very high probability.

More generally, it was shown in \cite{gowers:SzT} (in the case $G=\Z_N$
for $N$ prime, but the proof generalizes) that a set $A$ of
density $\alpha$ has about $\alpha^k|G|^2$ arithmetic progressions of length
$k$ if the characteristic function of $A$ is almost as small as it can
be, given its density, in a norm that is now called the
$U^{k-1}$-norm.  Green and Tao \cite{greentao:linprimes} have found the
most general statement that follows from the technique used to prove this
result, introducing a notion that they call the \emph{complexity}
of a system of linear forms. They prove that if $A$ has almost minimal
$U^{k+1}$-norm then it has the expected number of linear configurations 
of a given type, provided that the associated complexity is at most
$k$. The main result of this paper is that the converse is not true:
in particular there are certain systems of complexity 2 that are 
controlled by the $U^2$-norm, whereas the result of Green and Tao 
requires the stronger hypothesis of $U^3$-control. 

We say that a system of $m$ linear forms $L_1,\dots,L_m$ in $d$
variables has \emph{true complexity} $k$ if $k$ is the smallest
positive integer such that, for any set $A$ of density $\alpha$ and
almost minimal $U^{k+1}$-norm, the number of $d$-tuples
$(x_1,\dots,x_d)$ such that $L_i(x_1,\dots,x_d)\in A$ for every $i$
is approximately $\alpha^m|G|^d$. We conjecture that the true
complexity $k$ is the smallest positive integer $s$ for which
the functions $L_1^{s+1},\dots,L_m^{s+1}$ are linearly
independent. Using the ``quadratic Fourier analysis'' of Green and
Tao we prove this conjecture in the case where the complexity of
the system (in Green and Tao's sense) is 2, $s=1$ and $G$ is the
group $\F_p^n$ for some fixed odd prime $p$.
 
A closely related result in ergodic theory was recently proved 
independently by Leibman \cite{leibman:new}. We end the paper by
discussing the connections between his result and ours. 
\end{abstract}

\section{Introduction}

Let $A$ be a subset of a finite Abelian group $G$ and let
$\alpha=|A|/|G|$ be the \emph{density} of $A$. We say that $A$
is \emph{uniform} if it has one of several equivalent properties,
each of which says in its own way that $A$ ``behaves like a 
random set''. For example, writing $A$ for the characteristic
function of the set $A$, we can define the convolution $A*A$ by 
the formula
\[A*A(x)=\E_{y+z=x}A(y)A(z),\]
where the expectation is with respect to the uniform distribution
over all pairs $(y,z)\in G^2$ such that $y+z=x$; one of the 
properties in question is that the variance of $A*A$ 
should be small. If this is the case and $G$ has odd order, then it is easy 
to show that $A$ contains approximately $\alpha^3|G|^2$ triples of the 
form $(x,x+d,x+2d)$. Indeed, these triples are the solutions
$(x,y,z)$ of the equation $x+z=2y$, and
\[\E_{x+z=2y}A(x)A(y)A(z)=\E_yA*A(2y)A(y).\]
The mean of the function $A*A$ is $\alpha^2$, so if the variance
is sufficiently small, then the right-hand side is approximately
$\alpha^2\E_yA(y)=\alpha^3$. This is a probabilistic way of saying
that the number of solutions of $x+z=2y$ inside $A$ is approximately
$\alpha^3|G|^2$, which is what we would expect if $A$ was a random
set with elements chosen independently with probability~$\alpha$.

An easy generalization of the above argument shows that, given
any linear equation in $G$ of the form
\[c_1 x_1+ c_2 x_2+ \dots + c_m x_m =0,\] 
for suitable fixed coefficients $c_1, c_2, ..., c_m$, the
number of solutions in $A$ is approximately $\alpha^m|G|^{m-1}$.
Roughly speaking, you can choose $x_3,\dots,x_m$ in $A$ however 
you like, and if $A$ is sufficiently uniform then the number of 
ways of choosing $x_1$ and $x_2$ to lie in $A$ and satisfy the 
equation will almost always be roughly $\alpha^2|G|$. By ``suitable''
we mean that there are certain divisibility problems that must be
avoided. For example, if $G$ is the group $\F_2^n$, $x+z=2y$ 
and $x$ belongs to $A$, then $z$ belongs to $A$ for the
trivial reason that it equals $x$. Throughout this paper
we shall consider groups of the form $\F_p^n$ for some prime
$p$ and assume that $p$ is large enough for such problems not
to arise. 

When $k\geq 4$, uniformity of a set $A$ does not guarantee that $A$
contains approximately $\alpha^k|G|^2$ arithmetic progressions of
length $k$.  For instance, there are examples of uniform subsets of
$\Z_N$ that contain significantly more, or even significantly fewer
than, the expected number of four-term progressions
\cite{gowers:4APsEx}. It was established in \cite{gowers:SzT4} that
the appropriate measure for dealing with progressions of length 4 is a
property known as \emph{quadratic uniformity}: sets which are
sufficiently quadratically uniform contain roughly the correct number
of four-term progressions. We shall give precise definitions of
higher-degree uniformity in the next section, but for now let us
simply state the result, proved in \cite{gowers:SzT} in the case
$G=\Z_N$, that if $A$ is uniform of degree $k-2$, then $A$ contains
approximately $\alpha^k|G|^2$ arithmetic progressions of length
$k$. Moreover, if $A$ is uniform of degree $j$ for some $j<k-2$, 
then it does \emph{not} follow that $A$ must contain approximately 
$\alpha^k|G|^2$ arithmetic progressions of length $k$.

The discrepancy between $k$ and $k-2$ seems slightly unnatural until
one reformulates the statement in terms of solutions of equations. We
can define an arithmetic progression of length $k$ either as a
$k$-tuple of the form $(x,x+d,\dots,x+(k-1)d)$ or as a solution
$(x_1,x_2,\dots,x_k)$ to the system of $k-2$ equations
$x_i-2x_{i+1}+x_{i+2}=0$, $i=1,2,\dots,k-2$. In all the examples we
have so far discussed, we need uniformity of degree precisely $k$ in
order to guarantee approximately the expected number of solutions of a
system of $k$ equations. It is tempting to ask whether this is true in
general.

However, a moment's reflection shows that it is not. For example,
the system of equations $x_1-2x_2+x_3=0$, $x_4-2x_5+x_6=0$ has
about $\alpha^6|G|^4$ solutions in a uniform set, since the two
equations are completely independent. This shows that a sensible
conjecture must take account of how the equations interact with
each other.

A more interesting example is the system that consists of the
${m\choose 3}$ equations $x_{ij}+x_{jk}=x_{ik}$ in the 
${m\choose 2}$ unknowns $x_{ij}$, $1\leq i<j\leq m$. These
equations are not all independent, but one can of course choose an
independent subsystem of them. It is not hard to see that there is a
bijection between solutions of this system of equations where every
$x_{ij}$ belongs to $A$ and $m$-tuples $(x_1,\dots,x_m)$ such that
$x_j-x_i\in A$ whenever $i<j$. Now one can form a bipartite graph with
two vertex sets equal to $G$ by joining $x$ to $y$ if and only if
$y-x\in A$. It is well-known that if $A$ is uniform, then this
bipartite graph is quasirandom.  The statement that every $x_j-x_i$
belongs to $A$ can be reformulated to say that $(x_1,\dots,x_m)$ form
a clique in an $m$-partite graph that is built out of quasirandom
pieces derived from $A$. A ``counting lemma'' from the theory of
quasirandom graphs then implies easily that the number of such cliques
is approximately $\alpha^{m\choose 2}|G|^m$. So uniformity of degree 1 is
sufficient to guarantee that there are about the expected number of
solutions to this fairly complicated system of equations.

In their recent work on configurations in the primes, Green and Tao 
\cite{greentao:linprimes} analysed the arguments used to prove the above
results, which are fairly simple and based on repeated use of the
Cauchy-Schwarz inequality. They isolated the property that a system
of equations, or equivalently a system of linear forms, must have
in order for degree-$k$ uniformity to be sufficient for these
arguments to work, and called this property \emph{complexity}.
Since in this paper we shall have more than one notion of 
complexity, we shall sometimes call their notion 
\emph{Cauchy-Schwarz complexity}, or \emph{CS-complexity} for short.

\begin{definition}
Let $\Lsys=(L_1, ..., L_m)$ be a system of $m$ linear forms in $d$
variables. For $1\leq i \leq m$ and $s \geq 0$, we say that $\Lsys$ is
$s$-\emph{complex} at $i$ if one can partition the $m-1$ forms $\{L_j : j
\ne i\}$ into $s+1$ classes such that $L_i$ does not
lie in the linear span of any of these classes. The
\emph{Cauchy-Schwarz complexity} (or CS-complexity) of $\Lsys$ is
defined to be the least $s$ for which the system is $s$-complex at 
$i$ for all $1\leq i \leq m$, or $\infty$ if no such $s$ exists.
\end{definition}

To get a feel for this definition, let us calculate the complexity of
the system $\Lsys$ of $k$ linear forms $x,x+y,\dots,x+(k-1)y$. Any 
two distinct forms $x+iy$ and $x+jy$ in $\Lsys$ contain $x$ and $y$ in 
their linear span. Therefore, whichever form $L$ we take from $\Lsys$, 
if we wish to partition the others into classes that do not contain
$L$ in their linear span, then we must take these classes to be singletons.
Since we are partitioning $k-1$ forms, this tells us that the minimal
$s$ is $k-2$. So $\Lsys$ has complexity $k-2$.

Next, let us briefly look at the system $\Lsys$ of ${m\choose 2}$ forms 
$x_i-x_j$ ($1\leq i<j\leq m$) that we discussed above. If $L$ is
the form $x_i-x_j$ then no other form $L'\in\Lsys$ involves both
$x_i$ and $x_j$, so we can partition $\Lsys\setminus\{L\}$ into the 
forms that involve $x_i$ (which therefore do not involve $x_j$) and the 
forms that do not involve $x_i$. Since neither class includes $L$ in its
linear span, the complexity of $\Lsys$ is at most $1$. When $m\ge 3$ it 
can also be shown to be at least 1.

It follows from Green and Tao's result that if $A$ is sufficiently 
uniform and $\Lsys= (L_1, ..., L_m)$ has complexity at most 1, then 
$A$ contains approximately the expected number of $m$-tuples of the
form $(L_1(x_1, \dots, x_d),\dots,L_m(x_1, \dots x_d))$. (If the forms 
are defined over $\Z^d$, then this number is $\alpha^m|G|^d$.) 

Notice that this statement adequately explains all the cases we have
so far looked at in which uniformity implies the correct number of
solutions. It is thus quite natural to conjecture that Green and Tao's
result is tight. That is, one might guess that if the complexity
$\Lsys$ is greater than 1 then there exist sets $A$ that do not have
the correct number of images of $\Lsys$.

But is this correct? Let us look at what is known in the other 
direction, by discussing briefly the simplest example that shows that
uniform sets in $\Z_N$ do not have to contain the correct number
of arithmetic progressions of length 4. (Here we are taking $N$ to
be some large prime.) Roughly speaking, one takes $A$ to be the set 
of all $x$ such that $x^2$ mod $N$ is small. Then one makes use of 
the identity
\[x^2-3(x+d)^2+3(x+2d)^2-(x+3d)^2=0\]
to prove that if $x$, $x+d$ and $x+2d$ all lie in $A$, then $x+3d$
is rather likely to lie in $A$ as well, because $(x+3d)^2$ is a 
small linear combination of small elements of $\Z_N$. This means
that $A$ has ``too many'' progressions of length 4. (Later, we
shall generalize this example and make it more precise.)

The above argument uses the fact that the squares of the linear
forms $x$, $x+d$, $x+2d$ and $x+3d$ are linearly dependent. Later,
we shall show that if $\Lsys$ is \emph{any} system of linear forms
whose squares are linearly dependent, then essentially the same 
example works for $\Lsys$. This gives us a sort of ``upper bound''
for the set of systems $\Lsys$ that have approximately the right number
of images in any uniform set: because of the above example, we
know that the squares of the forms in any such system $\Lsys$ must be 
linearly independent. 

And now we arrive at the observation that motivated this paper: the
``upper bound'' just described does not coincide with the ``lower
bound'' of Green and Tao. That is, there are systems of linear forms
of complexity greater than 1 with squares that are linearly
independent. One of the simplest examples is the system
$(x,y,z,x+y+z,x+2y-z,x+2z-y)$. Another, which is
translation-invariant (in the sense that if you add a constant to
everything in the configuration, you obtain another configuration of
the same type), is $(x,x+y,x+z,x+y+z,x+y-z,x+z-y)$. A third and rather
natural example that is also translation-invariant is the configuration 
$$(x,x+a,x+b,x+c,x+a+b,x+a+c,x+b+c),$$
which can be thought of as a cube minus a point. All these examples 
have complexity 2, but it is not hard to produce
examples with arbitrarily high complexity.

In the light of such examples, we are faced with an obvious question:
which systems of linear forms have roughly the expected number of
images in any sufficiently uniform set? We conjecture that the
correct answer is given by the ``upper bound''---that is, that
square independence is not just necessary but also sufficient.
When the group $G$ is $\F_p^n$ for a fixed prime $p$, we prove 
this conjecture for systems of complexity 2. This includes 
the two examples above, and shows that having Cauchy-Schwarz
complexity at most 1 is not a necessary condition, even if it
is a natural sufficient one.

However, the proof is much deeper for systems of complexity 2.
Although the statement of our result is completely linear, we use
``quadratic Fourier analysis'', recently developed by Green and Tao
\cite{greentao:u3}, to prove it, and it seems that we are forced to do
so. Thus, it appears that Cauchy-Schwarz complexity captures the
systems for which an easy argument exists, while square independence
captures the systems for which the result is true.

Very recently, and independently, Leibman \cite{leibman:new} 
described a similar phenomenon in the ergodic-theoretic
context. In the final section of the paper we shall briefly 
outline how his results relate to ours. 

So far, we have concentrated on uniform sets. However, in the next
section we shall define higher-degree uniformity and formulate a more
complete conjecture, which generalizes the above discussion in a
straightforward way. Green and Tao proved that a system of
Cauchy-Schwarz complexity $k$ has approximately the correct number of
images in a set $A$ if $A$ is sufficiently uniform of degree
$k+1$. Once again, it seems that this is not the whole story, and that
the following stronger statement should be true: a linear system
$\Lsys=(L_1,\dots,L_m)$ has the right number of images in any set $A$
that is sufficiently uniform of degree $k$ if and only if the
functions $L_i^{k+1}$ are linearly independent.  The reason we have
not proved this is that the natural generalization of our existing
argument would have to use an as yet undeveloped general ``polynomial
Fourier analysis'', which is known only in the quadratic
case. However, it is easy to see how our arguments would generalize if such
techniques were available, which is compelling evidence that our conjecture 
(which we will state formally in a moment) is true.

\section{Uniformity norms and true complexity}

As promised, let us now give a precise definition of higher-degree 
uniformity. We begin by defining a sequence of norms, known as
\emph{uniformity norms}.

\begin{definition}
Let G be a finite Abelian group. For any positive integer $k \geq 2$
and any function $f: G \rightarrow \C$, define the \emph{$U^k$-norm}
by the formula
\[\|f\|_{U^k}^{2^k} := \E_{x,h_1, ..., h_k \in G} 
\prod_{\oomega \in \{0,1\}^k} C^{|\oomega|}f(x+\oomega\cdot \mathbf{h}),\]
where $\oomega\cdot\mathbf{h}$ is shorthand for $\sum_i\omega_ih_i$,
and $C^{|\oomega|}f=f$ if $\sum_i\omega_i$ is even and $\overline{f}$
otherwise.
\end{definition}

These norms were first defined in \cite{gowers:SzT} (in the case where
$G$ is the group $\Z_N$). Of particular interest in this paper will be
the $U^2$-norm and the $U^3$-norm. The former can be described in many
different ways.  The definition above expresses it as the fourth root
of the average of
\[f(x)\overline{f(x+h)}\overline{f(x+h')}f(x+h+h')\]
over all triples $(x,h,h')$. It is not hard to show that this
average is equal to $\|f*f\|_2^2$, and also to $\|\hat{f}\|_4^4$.
(These identities depend on appropriate normalizations---we follow
the most commonly used convention of taking averages in physical
space and sums in frequency space.)

We shall call a function $f$ $c$-\emph{uniform} if $\|f\|_{U^2}\le c$
and $c$-\emph{quadratically uniform} if $\|f\|_{U^3}\le c$. We shall
often speak more loosely and describe a function as uniform if it is 
$c$-uniform for some small $c$, and similarly for higher-degree uniformity.
We remark here that if $j\le k$ then $\|f\|_{U^j}\le\|f\|_{U^k}$, so
$c$-uniformity of degree $k$ implies $c$-uniformity of all lower degrees.

If $A$ is a subset of an Abelian group $G$ and the density of $A$ is
$\alpha$, then we say that $A$ is uniform of degree $k$ if it is close
in the $U^k$-norm to the constant function $\alpha$. More precisely,
we define the \emph{balanced function} $f(x)=A(x)-\alpha$ and say
that $A$ is $c$-\emph{uniform of degree} $k$ if $\|f\|_{U^k}\le c$.

The following theorem is essentially Theorem 3.2 in 
\cite{gowers:SzT}. (More precisely, in that paper the 
theorem was proved for the group $\Z_N$, but the proof 
is the same.)

\begin{theorem}\label{vn}
Let $k\geq 2$ and let $G$ be a finite Abelian group such that there
are no non-trivial solutions to the equation $jx=0$ for any $1\leq
j< k$. Let $c>0$ and let $f_1,f_2,\dots,f_k$ be functions from $G$
to $\C$ such that $\|f_i\|_\infty\leq 1$ for every $i$. Then
\[\Bigl|\E_{x,y\in G}f_1(x)f_2(x+y)\dots f_k(x+(k-1)y)\Bigr|
\leq\|f_k\|_{U^{k-1}}.\]
\end{theorem}

It follows easily from this result that if $A$ is a set of density
$\alpha$ and $A$ is $c$-uniform for sufficiently small $c$, then $A$
contains approximately $\alpha^k|G|^2$ arithmetic progressions of length
$k$. Very briefly, the reason for this is that we are trying to show that
the average
\[\E_{x,y}A(x)A(x+y)\dots A(x+(k-1)y)\]
is close to $\alpha^k$. Now this average is equal to
\[\E_{x,y}A(x)A(x+y)\dots f(x+(k-1)y)+
\alpha\E_{x,y}A(x)A(x+y)\dots A(x+(k-2)y).\]
The first of these terms is at most $c$, by Theorem \ref{vn}, and the
second can be handled inductively. The bound we obtain in this way
is $c(1+\alpha+\dots+\alpha^{k-1})\leq kc$.

We can now state formally Green and Tao's generalization in terms of
CS-complexity in the case where $G$ is the group $\Z_N$, which is
implicit in \cite{greentao:linprimes}.

\begin{theorem}\label{gvn}
Let $N$ be a prime, let $f_1,\dots,f_m$ be functions from 
$\Z_N$ to $[-1,1]$, 
and let $\Lsys$ be a linear system of CS-complexity $k$ consisting 
of $m$ forms in $d$ variables. Then, provided $N\geq k$,
\[ \Big|\E_{ x_1, ..., x_d \in \Z_N} \prod_{i=1}^m f(L_i(x_1, ..., x_d))\Big|
\leq \min_i\|f_i\|_{U^{k+1}}.\]
\end{theorem}

Just as in the case of arithmetic progressions, it follows easily
that if $A$ is a subset of $G$ of density $\alpha$, then the 
probability, given a random element $(x_1, ..., x_d)\in G^d$,
that all the $m$ images $L_i(x_1, ..., x_d)$ lie in $A$ is 
approximately $\alpha^m$. (The inductive argument depends on 
the obvious fact that if $\Lsys$ has complexity at most $k$
then so does any subsystem of $\Lsys$.)

Green and Tao proved the above theorem because they were investigating
which linear configurations can be found in the primes. For that
purpose, they in fact needed a more sophisticated ``relative'' version
of the statement. Since the proof of the version we need here is
simpler (partly because we are discussing systems of complexity at
most 2, but much more because we do not need a relative version), 
we give it for the convenience of the reader. This is another result
where the proof is essentially the same for all Abelian groups, give
or take questions of small torsion. Since we need it in the case 
$G=\F_p^n$, we shall just prove it for this group. The reader should
bear in mind that for this group, one should understand linear 
independence of a system of forms as independence over $\F_p$ when
one is defining complexity (and also square-independence).

The first step of Green and Tao's proof was to put an arbitrary
linear system into a convenient form for proofs. Given a linear
form $L$ in $d$ variables $x_1,\dots,x_d$, let us define the
\emph{support} of $L$ to be the set of $j$ such that $L$ depends
on $x_j$. That is, if $L(x_1,\dots,x_d)=\lambda_1x_1+\dots+\lambda_dx_d$
then the support of $L$ is $\{i:\lambda_i\ne 0\}$. Let 
$\Lsys=(L_1,\dots,L_m)$ be a system of linear forms and let the
support of $L_i$ be $\sigma_i$ for each $i$. Then $\Lsys$ is said
to be in $s$-\emph{normal form} if it is possible to find subsets
$\tau_i\subset\sigma_i$ for each $i$ with the following two properties.

(i) Each $\tau_i$ has cardinality at most $s+1$.

(ii) If $i\ne j$ then $\tau_i$ is not a subset of $\sigma_j$.

If a linear system $\Lsys$ is in $s$-normal form, then it has 
complexity at most $s$. Indeed, if $\tau_i$ has $r$ elements
$\{i_1,\dots,i_r\}$, then one can partition the remaining forms
into $r$ sets $\Lsys_1,\dots,\Lsys_r$ in such a way that no form
in $\Lsys_h$ uses the variable $x_{i_h}$. Since $L_i$ \emph{does}
use the variable $x_{i_h}$ it is not in the linear span of $\Lsys_h$.

The converse of this statement is false, but Green and Tao prove that
every linear system of complexity $s$ can be ``extended'' to one that
is in $s$-normal form. This part of the proof is the same in both
contexts, so we do not reproduce it. All we need to know here is that
if we prove Theorem \ref{gvn} for systems in normal form then we have
it for general systems.

Just to illustrate this, consider the obvious system associated
with arithmetic progressions of length 4, namely $(x,x+y,x+2y,x+3y)$.
This is not in 2-normal form, because the support of the first form
is contained in the supports of the other three. However, the system
$(-3x-2y-z,-2x-y+w,-x+z+2w,y+2z+3w)$ \emph{is} in 2-normal form (since
the supports have size 3 and are distinct) and its images are also
uniformly distributed over all arithmetic progressions of length 4
(if we include degenerate ones).

Now let us prove Theorem \ref{gvn} when $k=2$. Without loss of generality 
we may assume that $\Lsys$ is in
$2$-normal form at $1$, and that it is the only form using all three
variables $x_1=x, x_2=y$ and $x_3=z$. We use the shorthand $h(x,y,z)=
f(L_1(x_1, x_2, ..., x_d))$, and denote by $b(x,y)$ any general bounded
function in two variables $x$ and $y$. It is then possible to rewrite
\[ \E_{x_1, ..., x_d \in \F_p^n} \prod_{i=1}^m f(L_i(x_1, ..., x_d))\]
as
\[\E_{x_4, x_5, ..., x_d} \E_{x,y,z} h(x,y,z)b(x,y)b(y,z)b(x,z).\]
Here, the functions $h$ and $b$ depend on the variables $x_4,\dots,x_d$
but we are suppressing this dependence in the notation.

Estimating the expectation over $(x,y,z)$ is a well-known argument
from the theory of quasirandom hypergraphs. (See for instance Theorem 
4.1 in \cite{gowers:HRL3}.) First, we apply Cauchy-Schwarz
and use the boundedness of $b$ to obtain an upper bound of
\[(\E_{x,y} (\E_z h(x,y,z) b(x,z) b(y,z))^2)^{1/2}.\]
Expanding out the square and rearranging yields
\[(\E_{y,z,z'}b(y,z)b(y,z') \E_x h(x,y,z)h(x,y,z') b(x,z) b(x,z'))^{1/2},\]
and by a second application of Cauchy-Schwarz we obtain an upper bound of
\[(\E_{y,z,z'}(\E_x h(x,y,z)h(x,y,z') b(x,z)b(x,z'))^2)^{1/4}.\]
A second round of interchanging summation followed by a third
application of Cauchy-Schwarz gives us an upper bound of
\[(\E_{x,x',z,z'} (\E_y h(x,y,z)h(x,y,z')h(x',y,z)h(x',y,z'))^2)^{1/8}.\]
This expression equals the ``octahedral norm'' of the function 
$h(x,y,z)$---a hypergraph analogue of the $U^3$-norm. Because 
for fixed $x_4,\dots,x_d$, $h$ depends only on the linear 
expression $L_1(x,y,z)$, a simple change of variables can be
used to show that it is in fact equal to $\|f\|_{U^3}$.

Now all that remains is to take the expectation over the remaining
variables and the proof is complete. It is also not hard to generalize
to arbitrary $k$, but this we leave as an exercise to the reader.

Now, as we stated earlier, Theorem \ref{gvn} does not settle the question
of which systems are controlled by which degrees of uniformity. 
Accordingly, we make the following definition.

\begin{definition}
Let $\Lsys$ be a system of $m$ distinct linear forms $L_1,L_2,\dots,L_m$
in $d$ variables. The \emph{true complexity} of $\Lsys$ is the smallest
$k$ with the following property. For every $\epsilon>0$ there exists
$\delta>0$ such that if $G$ is any finite Abelian group and $f:G\rightarrow\C$
is any function with $\|f\|_\infty\leq 1$ and $\|f\|_{U^{k+1}}\leq\delta$,
then
\[\Bigl|\E_{ x_1, ..., x_d \in G} \prod_{i=1}^m 
f(L_i(x_1, ..., x_d))\Bigr|\leq\epsilon.\]
\end{definition}

The main conjecture of this paper is now simple to state precisely.

\begin{conjecture}\label{mainconj}
The true complexity of a system of linear forms $\Lsys=(L_1,\dots,L_m)$
is equal to the smallest $k$ such that the functions $L_i^{k+1}$ are
linearly independent.
\end{conjecture}

In the next section, we shall prove this conjecture in the simplest
case that is not covered by the result of Green and Tao, namely the
case when $k=1$ and $\Lsys$ has CS-complexity 2. All other cases would
require a more advanced form of polynomial Fourier analysis than the
quadratic Fourier analysis that is so far known, but we shall explain
why it will almost certainly be possible to generalize our argument
once such a theory is developed.

\section{True complexity for vector spaces over finite fields}

We shall now follow the course that is strongly advocated by Green
\cite{green:finitemodels} and restrict attention to the case where
$G$ is the group $\F_p^n$, where $p$ is a fixed prime and $n$ tends
to infinity. The reason for this is that it makes many arguments
technically simpler than they are for groups with large torsion
such as $\Z_N$. In particular, one can avoid the technicalities
associated with Bohr sets. These arguments can then almost always
be converted into more complicated arguments for $\Z_N$. (In a
forthcoming paper, we give a different proof for the case
$\F_p^n$ and carry out the conversion process. That proof is
harder than the proof here but gives significantly
better bounds and is easier to convert.)

We begin this section with the easier half of our argument, showing
that if $\Lsys$ is a system of linear forms $(L_1,\dots,L_m)$ and
if there is a linear dependence between the squares of  
these forms, then the true complexity of $\Lsys$ is greater than 1.
This part can be proved almost as easily for $\Z_N$, but we shall
not do so here.

\subsection{Square-independence is necessary}

Let us start by briefly clarifying what we mean by square-independence
of a linear system $\Lsys=(L_1,\dots,L_m)$. When the group $G$ is $\Z_N$,
then all we mean is that the functions $L_i^2$ are linearly independent,
but when it is $\F_p^n$, then this definition does not make sense any
more. Instead, we ask for the quadratic forms $L_i^TL_i$ to be linearly
independent. If $L_i(x_1,\dots,x_d)=\sum_r\gamma_r^{(i)}x_r$, then 
$L_i^TL_i(x_1,\dots,x_d)=\sum_r\sum_s\gamma_r^{(i)}\gamma_s^{(i)}x_rx_s$.
Therefore, what we are 
interested in is linear independence of the matrices
$\Gamma_{rs}^{(i)}=\gamma_r^{(i)}\gamma_s^{(i)}$ over $\F_p$. (Note 
that in the case of $\Z_N$, this is equivalent to independence of the 
functions $L_i^2$.)

\begin{theorem}\label{badex}
Let $\Lsys=(L_1,\dots,L_m)$ be a system of linear forms in $d$ variables
and suppose
that the quadratic forms $L_i^TL_i$ are linearly dependent over $\F_p$.
Then there exists $\epsilon>0$ such that for every $\delta>0$ there exists
$n$ and a set $A\subset \F_p^n$ with the following two properties.

(i) $A$ is $\delta$-uniform of degree $1$.

(ii) If $\x=(x_1,\dots,x_d)$ is chosen randomly from $(\F_p^n)^d$, then
the probability that $L_i(\x)$ is in $A$ for every $i$ is at least
$\alpha^m+\epsilon$, where $\alpha$ is the density of $A$.

In other words, the true complexity of $\Lsys$ is at least 2.
\end{theorem}

For the proof we require the following standard lemma, which
says that certain Gauss sums are small. A proof can be found in 
\cite{green:montreal}, for example. 

\begin{lemma}\label{gauss}
Suppose that $q: \F_p^n \rightarrow \F_p$ is a quadratic form of rank
$r$. That is, suppose that $q(x)=x^T M x +b^T x$ for some matrix
$M$ of rank $r$ and some vector $b \in \F_p^n$. Then
\[|\E_{x \in \F_p^n}\; \omega^{q(x)}| \leq p^{-r/2},\]
with equality if $b=0$. In particular, 
\[|\E_{x \in \F_p^n}\; \omega^{\eta x^Tx}| \leq p^{-n/2}\]
for any non-zero $\eta \in \F_p$.
\end{lemma}

\begin{proof}[Proof of Theorem \ref{badex}]
Let $A$ be the set $\{x\in\F_p^n:x^Tx=0\}$. Then the characteristic
function of $A$ can be written as
\[A(x)=\E_u\omega^{ux^Tx},\]
where $\omega=\exp(2\pi i/p)$ and the expectation is taken over
$\F_p$. Let us now take any square-\emph{independent} system 
$\Lsys=(L_1,\dots,L_m)$ of linear forms in $\x=(x_1,\dots,x_d)$ 
and estimate the expectation $\E_\x\prod_iA(L_i(\x))$.

Using the formula for $A(x)$, we can rewrite this expectation as
\[\E_{\x\in(\F_p^n)^d}\E_{u_1,\dots,u_m\in\F_p}
\omega^{\sum_iu_iL_i(\x)^TL_i(\x)}.\]
We can break this up into $p^m$ expectations over $\x$, one for each
choice of $u_1,\dots,u_m$. 

If the $u_i$ are all zero, then the expectation over $\x$ is just the
expectation of the constant function 1, so it is 1. Otherwise, since
the quadratic forms $L_i^TL_i$ are linearly independent, the sum
$\sum_iu_iL_i(\x)^TL_i(\x)$ is a non-zero quadratic form 
$q(x)=\sum_{i,j}\gamma_{ij}x_i^Tx_j$.

Without loss of generality, there exists $j$ such that $\gamma_{1j}\ne
0$.  If in addition $\gamma_{11}=0$, then for every choice of
$x_2,\dots,x_d$ we can write $q(x)$ in the form $r^Tx_1+z$, where
$r=\sum_j\gamma_{1j}x_j$ and $z$ depends on $x_2,\dots,x_d$ only.
This is a non-constant linear function of $x_1$ except when
$\sum_j\gamma_{1j}x_j=0$. Since not every $\gamma_{1j}$ is zero, 
this happens with probability $p^{-n}$. Therefore, $|\E_\x\omega^{q(\x)}|\leq
p^{-n}$ in this case. 
If $\gamma_{11}\ne 0$, then
this same function has the form $\gamma_{11}x_1^Tx_1+r^Tx_1$ for some
element $r\in\F_p^n$ (which depends on $x_2,\dots,x_d$).  In this
case, Lemma \ref{gauss} implies that the expectation is at most
$p^{-n/2}$.

Since the probability that $u_1=\dots=u_m=0$ is $p^{-m}$, this shows
that 
\[\Bigl|\E_{\x \in (\F_p^n)^d} \prod_iA(L_i(\x))-p^{-m}\Bigr|\le p^{-n/2}.\]

Applying this result in the case where $\Lsys$ consists of the single
form $x$, we see that the density of $A$ differs from $p^{-1}$ by
at most $p^{-n/2}$. Therefore, we have shown that for this particular
set $A$, square-independence of $\Lsys$ guarantees approximately 
the ``correct'' probability that every $L_i(\x)$ lies in $A$.

This may seem like the opposite of what we were trying to prove, but
in fact we have almost finished, for the following simple reason.  If
we now take $\Lsys$ to be an arbitrary system $(L_1,\dots,L_m)$ of
linear forms, then we can choose from it a maximal square-independent
subsystem. Without loss of generality this subsystem is
$(L_1,\dots,L_l)$.  Then all the quadratic forms $L_i^TL_i$ with $i>l$
are linear combinations of $L_1^T L_1,\dots,L_l^T L_l$, so a
sufficient condition for every $L_i^TL_i(\x)$ to be zero is that it is
zero for every $i\leq l$. But this we know happens with probability
approximately $p^{-l}$ by what we have just proved.  Therefore, if
$\Lsys$ is not square-independent, then $A^m$ contains ``too many''
$m$-tuples of the form $(L_1(\x),\dots,L_m(\x))$.
\end{proof}

\subsection{A review of quadratic Fourier analysis}

We shall now turn our attention to the main result of this paper,
which states that if $\Lsys$ has CS-complexity at most 2 and is
square-independent, then the true complexity of $\Lsys$ is at most
1. We begin with a quick review of quadratic Fourier analysis for
functions defined on $\F_p^n$. Our aim in this review is to give
precise statements of the results that we use in our proof. The reader
who is prepared to use quadratic Fourier analysis as a black box
should then find that this paper is self-contained.

So far in our discussion of uniformity, we have made no mention of
Fourier analysis at all. However, at least for the $U^2$-norm, there 
is a close connection. Let $f$ be a complex-valued function defined on a
finite Abelian group $G$. If $\chi$ is a character on $G$, the Fourier
coefficient $\hat{f}(\chi)$ is defined to be $\E_xf(x)\chi(x)$. The
resulting Fourier transform satisfies the convolution identity
$\widehat{f*g}=\hat{f}\hat{g}$, Parseval's identity
$\|\hat{f}\|_2=\|f\|_2$ and the inversion formula
$f(x)=\sum_\chi\hat{f}(\chi)\chi(-x)$.  (The second and third
identities depend on the correct choice of normalization: $\|f\|_2^2$
is defined to be $\E_x|f(x)|^2$, whereas $\|\hat{f}\|_2^2$ is defined
to be $\sum_\chi|\hat{f}(\chi)|^2$.  That is, as mentioned earlier, we
take averages in $G$ and sums in $\hat{G}$.) It follows that
$\|f\|_{U^2}^4=\|\hat{f}\|_4^4$, since both are equal to
$\|f*f\|_2^2$.

It is often useful to split a function $f$ up into a ``structured''
part and a uniform part. One way of doing this is to let $K$ be
the set of all characters $\chi$ for which $|\hat{f}(\chi)|$ is 
larger than some $\delta$ and to write $f=f_1+f_2$, where
$f_1=\sum_{\chi\in K}\hat{f}(\chi)\chi(-x)$ and 
$f_2=\sum_{\chi\notin K}\hat{f}(\chi)\chi(-x)$. If $\|f\|_\infty\le 1$,
(as it is in many applications), then Parseval's identity implies
that $|K|\le\delta^{-2}$, and can also be used to show that
$\|f_2\|_{U^2}\le\delta^{1/2}$. That is, $K$ is not too large,
and $f_2$ is $\delta^{1/2}$-uniform.

When $G$ is the group $\F_p^n$, the characters all have the
form $x\mapsto\omega^{r^Tx}$. Notice that this character is
constant on all sets of the form $\{x:r^Tx=u\}$, and that these
sets partition $\F_p^n$ into $p$ affine subspaces of codimension
1. Therefore, one can partition $\F_p^n$ into at most $p^{|K|}$
affine subspaces of codimension $|K|$ such that $f_1$ is 
constant on each of them. This is the sense in which $f_1$
is ``highly structured''.

The basic aim of quadratic Fourier analysis is to carry out a
similar decomposition for the $U^3$-norm. That is, given a function
$f$, we would like to write $f$ as a sum $f_1+f_2$, where $f_1$ is
``structured'' and $f_2$ is \emph{quadratically} uniform. Now this
is a stronger (in fact, much stronger) property to demand of $f_2$,
so we are forced to accept a weaker notion of structure for $f_1$.

Obtaining any sort of structure at all is significantly harder than
it is for the $U^2$-norm, and results in this direction are much 
more recent. The first steps were taken in \cite{gowers:SzT4} and
\cite{gowers:SzT} for the group $\Z_N$ in order to give an 
analytic proof of Szemer\'edi's
theorem. The structure of that proof was as follows: Theorem \ref{vn}
(of the present paper) can be used to show that if a set $A$ is sufficiently
uniform of degree $k-2$, then it must contain an arithmetic progression
of length $k$. Then an argument that is fairly easy when $k=3$ but
much harder when $k\geq 4$ can be used to show that if $A$ is
\emph{not} $c$-uniform of degree $k$, then it must have
``local correlation'' with a function of the form $\omega^{\phi(x)}$,
where $\omega=\exp{2\pi i/N}$ and $\phi$ is a polynomial of degree
$d$. ``Local'' in this context means that one can partition $\Z_N$
into arithmetic progressions of size $N^\eta$ (for some $\eta$ that
depends on $c$ and $k$ only) on a large proportion of which one can 
find such a correlation.

This was strong enough to prove Szemer\'edi's theorem, but for several
other applications the highly local nature of the correlation is too weak. 
However, in the quadratic case,
this problem has been remedied by Green and Tao \cite{greentao:u3}. In
this case, the obstacle to ``globalizing'' the argument is that a
certain globally-defined bilinear form that occurs in the proof of
\cite{gowers:SzT} is not symmetric, and thus does not allow one to
define a corresponding globally-defined quadratic form. (In the
context of $\Z_N$, ``global'' means something like ``defined on a
proportional-sized Bohr set''. For $\F_p^n$ one can take it to mean
``defined everywhere''.) Green and Tao discovered an
ingenious ``symmetry argument'' that allows one to replace the
bilinear form by one that \emph{is} symmetric, and this allowed them
to prove a quadratic structure theorem for functions with large
$U^3$-norm that is closely analogous to the linear structure theorem
that follows from conventional Fourier analysis.

An excellent exposition of this structure theorem when the group
$G$ is a vector space over a finite field can be found in 
\cite{green:montreal}. This contains proofs of all the background
results that we state here.

Recall that in the linear case, we called $f_1$ ``structured'' because
it was constant on affine subspaces of low codimension.  For quadratic
Fourier analysis, we need a quadratic analogue of the notion of a
decomposition of $\F_p^n$ into parallel affine subspaces of
codimension $d_1$. In order to define such a decomposition, one can take
a surjective linear map $\Gamma_1:\F_p^n\rightarrow\F_p^{d_1}$ and for
each $a\in\F_p^{d_1}$ one can set $V_a$ to equal $\Gamma_1^{-1}(\{a\})$.
If we want to make this idea quadratic, we should replace the 
linear map $\Gamma_1$ by a ``quadratic map'' $\Gamma_2$, which we do
in a natural way as follows. We say that a function
$\Gamma_2:\F_p^n\rightarrow\F_p^{d_2}$ is \emph{quadratic} if it is
of the form $x\mapsto(q_1(x),\dots,q_{d_2}(x))$, where $q_1,\dots,q_{d_2}$
are quadratic forms on $\F_p^n$. Then, for each $b\in\F_p^{d_2}$ we
define $W_b$ to be $\{x\in\F_p^n:\Gamma_2(x)=b\}$.

In \cite{greentao:r4boundsI}, Green and Tao define $\bone$ to be the algebra
generated by the sets $V_a$ and $\btwo$ for the finer algebra 
generated by the sets $V_a\cap W_b$. They call $\bone$ a 
\emph{linear factor of complexity} $d_1$ and $(\bone,\btwo)$ a 
\emph{quadratic factor of complexity} $(d_1,d_2)$. This is to draw
out a close analogy with the ``characteristic factors'' that 
occur in ergodic theory. 

These definitions give us a suitable notion of a ``quadratically
structured'' function---it is a function $f_1$ for which we can
find a linear map $\Gamma_1:\F_p^n\rightarrow\F_p^{d_1}$ and
a quadratic map $\Gamma_2:\F_p^n\rightarrow\F_p^{d_2}$ such that
$d_1$ and $d_2$ are not too large and $f_1$ is constant on the
sets $V_a\cap W_b$ defined above. This is equivalent to saying that
$f_1$ is measurable with respect to the algebra $\btwo$, and also
to saying that $f_1(x)$ depends on $(\Gamma_1(x),\Gamma_2(x))$ only.

The quadratic structure theorem of Green and Tao implies that a bounded
function $f$ defined on $\F_p^n$ can be written as a sum $f_1+f_2$,
where $f_1$ is quadratically structured in the above sense, and
$\|f_2\|_{U^3}$ is small.  In \cite{greentao:r4boundsI} the result is
stated explicitly for $p=5$, but this is merely because of the
emphasis placed on 4-term progressions. The proof is not affected by
the choice of $p$ (as long as it stays fixed).

In the statement below, we write $\E(f|\btwo)$ for the conditional
expectation, or averaging projection, of $f$. That is, if $X=V_a\cap W_b$
is an atom of $\btwo$ and $x\in X$, then $\E(f|\btwo)(x)$ is the
average of $f$ over $X$. Since the function $\E(f|\btwo)$ is constant
on the sets $V_a\cap W_b$, it is quadratically structured in the
sense that interests us.

\begin{theorem}\cite{greentao:r4boundsI}\label{qd1}
Let $p$ be a fixed prime, let $\delta >0$ and suppose that
$n > n_0(\delta)$ is sufficiently large. Given any function $f: \F_p^n
\rightarrow [-1,1]$, there exists a quadratic factor $(\bone,\btwo)$
of complexity at most $((4 \delta^{-1})^{3C_0+1}, (4
\delta^{-1})^{2C_0+1})$ together with a decomposition
\[f=f_1+f_2,\]
where
\[f_1 := \E(f|\btwo) \;\;\mbox{ and }\;\; \|f_2\|_{U^3} \leq \delta.\]
The absolute constant $C_0$ can be taken to be $2^{16}$.
\end{theorem}

As it stands, the above theorem is not quite suitable for applications,
because technical problems arise if one has to deal with quadratic
forms of low rank. (Notice that so far we have said nothing about
the quadratic forms $q_i$---not even that they are distinct.)
Let $\Gamma_2=(q_1,\dots,q_k)$ be a quadratic map and for each $i$
let $\beta_i$ be the symmetric bilinear form corresponding to $q_i$:
that is, $\beta_i(x,y)=(q_i(x+y)-q_i(x)-q_i(y))/2$. We shall say that 
$\Gamma_2$ \emph{is of rank at least} $r$ if the bilinear form
$\sum_i \lambda_i \beta_i$ has rank at least $r$ whenever 
$\lambda_1,\dots,\lambda_{d_2}$ are elements of $\F_p$ that are 
not all zero. If $\Gamma_2$ is used in combination with some 
linear map $\Gamma_1$ to define a quadratic factor $(\bone,\btwo)$, 
then we shall also say that this quadratic factor has rank at least $r$. 

Just to clarify this definition, let us prove a simple lemma that
will be used later.

\begin{lemma}\label{rankrestr}
Let $\beta$ be a symmetric bilinear form of rank $r$ on $\F_p^n$ 
and let $W$ be a subspace of $\F_p^n$ of codimension $d_1$. Then 
the rank of the restriction of $\beta$ to $W$ is at least $r-2d_1$.
\end{lemma}

\begin{proof} 
Let $V=\F_p^n$. For every subspace $W$ of $V$, let us write
$W^\perp$ for the subspace 
\[\{v\in V:\beta(v,w)=0\ \mbox{for every}\ w\in W\}.\]
Let us define the \emph{nullity} of $\beta$ to be the
dimension of $V^\perp$. Then the rank of $\beta$ is
equal to $n$ minus its nullity, which is the codimension
of $V^\perp$. We are assuming that this is $r$.

Now let $W$ have codimension $d_1$. We begin by bounding from above
the dimension of $W^\perp$. To do this, let $Y$ be a complement for 
$W$, which, by hypothesis, will have dimension $d_1$. Then 
$V^\perp=W^\perp\cap Y^\perp$ and $Y^\perp$ has dimension at 
least $n-d_1$, so the dimension of $W^\perp$ is at most 
$d_1+\dim(V^\perp)$, which, by hypothesis, is at least 
$d_1+n-r$. Therefore, the codimension of $W^\perp$ is at
most $r-d_1$, which implies that the codimension of 
$W^\perp$ inside $W$ is at most $r-2d_1$. This implies
the result. 
\end{proof}

We are now in a position to state the version of the structure theorem
that we shall be using. It can be read out of (but is not explicitly
stated in) \cite{green:montreal} and \cite{greentao:r4boundsI}.

\begin{theorem}\label{qd2}
Let $p$ be a fixed prime, let $\delta >0$, let $r: \N \rightarrow
\N$ an arbitrary function (which may depend on $\delta$) and
suppose that $n > n_0(r, \delta)$ is sufficiently large. Then given
any function $f: \F_p^n \rightarrow [-1,1]$, there exists
$d_0=d_0(r,\delta)$ and a quadratic factor $(\bone,\btwo)$ of rank at
least $r(d_1+d_2)$ and complexity at most $(d_1,d_2), d_1, d_2 \leq
d_0$, together with a decomposition
\[f=f_1+f_2+f_3,\]
where
\[f_1 := \E(f|\btwo),\;\;\; \|f_2\|_2 \leq \delta \;\;\mbox{ and }\;\; 
\|f_3\|_{U^3} \leq \delta.\]
\end{theorem}

Note that $\E f_1=\E f$. In particular $\E f_1=0$ whenever $f$ is the
balanced function of a subset of $\F_p^n$. It can be shown that 
$f_1$ is uniform whenever $f$ is uniform: roughly speaking, the reason
for this is that $\E(f|\bone)$ is approximately zero and the atoms of
$\btwo$ are uniform subsets of the atoms of $\bone$. However, we shall
not need this fact. 

We shall apply Theorem \ref{qd2} when $r$ is the function $d\mapsto 2md+C$
for a constant $C$. Unfortunately, ensuring that factors have high rank
is an expensive process: even for this modest function the argument 
involves an iteration that increases $d_0$ exponentially at every step.
For this reason we have stated the theorem in a qualitative way. A
quantitative version would involve a tower-type bound.

\subsection{Square-independence is sufficient} 
We now have the tools we need to show
that square-independence coupled with CS-complexity $2$ is sufficient
to guarantee the correct number of solutions in uniform sets. The basic
idea of the proof is as follows. Given a set $A\subset\F_p^n$ of density
$\alpha$, we first replace it by its balanced function $f(x)=A(x)-\alpha$.
Given a square-independent linear system $\Lsys$ of complexity at most 2, 
our aim is to show, assuming that $\|f\|_{U^2}$ is sufficiently small, 
that 
\[\E_{\x\in(\F_p^{n})^d}\prod_{i=1}^mf(L_i(\x))\]
is also small. (Once we have done that, it will be straightforward to
show that the same average, except with $A$ replacing $f$, is close to 
$\alpha^m$.) In order to carry out this estimate, we first apply the
structure theorem to decompose $f$ as $f_1+f_2+f_3$, where $f_1$
is quadratically structured, $f_2$ is small in $L_2$ and $f_3$ is
quadratically uniform. This then allows us to decompose the product 
into a sum of $3^m$ products, one for each way of choosing $f_1$, 
$f_2$ or $f_3$ from each of the $m$ brackets. If
we ever choose $f_2$, then the Cauchy-Schwarz inequality implies
that the corresponding term is small, and if we ever choose $f_3$
then a similar conclusion follows from Theorem \ref{gvn}. Thus, the most
important part of the proof is to use the linear uniformity and
quadratic structure of $f_1$ to prove that the product 
\[\E_{\x\in(\F_p^{n})^d}\prod_{i=1}^mf_1(L_i(\x))\]
is small. This involves a calculation that generalizes the one
we used to prove Theorem \ref{badex}. The main step is the following lemma,
where we do the calculation in the case where the linear factor
$\bone$ is trivial.

\begin{lemma}\label{quadfactor}
Let $\Lsys=(L_1,\dots,L_m)$ be a square-independent system of linear
forms and let $\Gamma_2=(q_1,\dots,q_{d_2})$ be a quadratic map from
$\F_p^n$ to $\F_p^{d_2}$ of rank at least $r$. Let $\phi_1,\dots,\phi_m$
be linear maps from $(\F_p^n)^d$ to $\F_p^{d_2}$ and let $b_1,\dots,b_m$ 
be elements of $\F_p^{d_2}$. Let $\x=(x_1,\dots,x_d)$ be a randomly
chosen element of $(\F_p^n)^d$. Then the probability that
$\Gamma_2(L_i(\x))=\phi_i(\x)+b_i$ for every $i$ differs from 
$p^{-md_2}$ by at most $p^{-r/2}$.
\end{lemma}

\begin{proof}
Let $\Lambda$ be the set of all $m\times d_2$ matrices
$\lambda=(\lambda_{ij})$ over $\F_p$ and let us write
$\phi_i=(\phi_{i1},\dots,\phi_{id_2})$ and 
$b_i=(b_{i1},\dots,b_{id_2})$ for each $i$. The probability
we are interested in is the probability that 
$q_j(L_i(\x))=\phi_{ij}(\x)+b_{ij}$
for every $i\leq m$ and every $j\leq d_2$. This equals
\[\E_\x\E_{\lambda\in\Lambda}\prod_{i=1}^m\prod_{j=1}^{d_2}
\omega^{\lambda_{ij}(q_j(L_i(\x))-\phi_{ij}(\x)-b_{ij})},\]
since if $q_j(L_i(\x))=\phi_{ij}(\x)+b_{ij}$ for every $i$ and 
$j$, then the expectation over $\lambda$ is 1, and otherwise if we 
choose $i$ and $j$ such that $q_j(L_i(\x))\ne\phi_{ij}(\x)+b_{ij}$ 
and consider the expectation over $\lambda_{ij}$ while all other entries
of $\lambda$ are fixed, then we see that the expectation over
$\lambda$ is zero.

We can rewrite the above expectation as
\[\E_{\lambda\in\Lambda}\E_\x
\omega^{\sum_{i,j}\lambda_{ij}(q_j(L_i(\x))-\phi_{ij}(\x)-b_{ij})}.\]
If $\lambda=0$, then obviously the expectation over $\x$
is 1. This happens with probability $p^{-md_2}$.
Otherwise, for each $i$ let us say that the coefficients 
of $L_i$ are $c_{i1},\dots,c_{id}$. That is, let 
$L_i(\x)=\sum_{u=1}^d c_{iu}x_u$. Then 
\[q_j(L_i(\x))=\sum_{u,v}c_{iu}c_{iv}\beta_j(x_u,x_v),\]
where $\beta_j$ is the bilinear form associated with $q_j$.
Choose some $j$ such that $\lambda_{ij}$ is non-zero for
at least one $i$. Then the square-independence of the linear
forms $L_i$ implies that there exist $u$ and $v$ such that
$\sum_i\lambda_{ij}c_{iu}c_{iv}$ is not zero.

Fix such a $j$, $u$ and $v$ and do it in such a way that
$u=v$, if this is possible.
We shall now consider the expectation as $x_u$ and $x_v$ vary
with every other $x_w$ fixed. Notice first that
\[\sum_{i,j}\lambda_{ij}q_j(L_i(\x))=\sum_{i,j}\sum_{t,w}
\lambda_{ij}c_{it}c_{iw}\beta_j(x_t,x_w).\]
Let us write $\beta_{tw}$ for the bilinear form
$\sum_{i,j}\lambda_{ij}c_{it}c_{iw}\beta_j$, so that
this becomes $\sum_{t,w}\beta_{tw}(x_t,x_w)$. Let us
also write $\phi(\x)$ for $\sum_{ij}\lambda_{ij}\phi_{ij}(\x)$
and let $\phi_1,\dots,\phi_d$ be linear maps from $\F_p^n$ to
$\F_p$ such that $\phi(\x)=\sum_t\phi_t(x_t)$ for every $\x$. 
Then
\[\sum_{i,j}\lambda_{ij}(q_j(L_i(\x))-\phi_{ij}(\x))
=\sum_{t,w}\beta_{tw}(x_t,x_w)-\sum_t\phi_t(x_t).\]
Notice that if we cannot get $u$ to equal $v$, then 
$\sum_i\lambda_{ij}c_{iu}^2=0$ for every $u$ and every $j$,
which implies that $\beta_{uu}=0$. Notice also that the
assumption that $\Gamma_2$ has rank at least $r$ and the
fact that $\sum_i\lambda_{ij}c_{iu}c_{iw}\ne 0$ for at 
least one $j$ imply that $\beta_{uv}$ has rank at least $r$.

If we fix every $x_t$ except for $x_u$ and $x_v$, then 
$\sum_{t,w}\beta_{tw}(x_t,x_w)-\sum_t\phi_t(x_t)$ is a 
function of $x_u$ and $x_v$ of the form
\[\beta_{uv}(x_u,x_v)+\psi_u(x_u)+\psi_v(x_v),\]
where $\psi_u$ and $\psi_v$ are linear functionals on $\F_p^n$
(that depend on the other $x_t$). 

Now let us estimate the expectation
\[\E_{x_u,x_v}\omega^{\sum_{i,j}\lambda_{ij}
(q_j(L_i(\x))-\phi_{ij}(\x)-b_{ij})},\]
where we have fixed every $x_t$ apart from $x_u$ and $x_v$.
Letting $b=\sum\lambda_{ij}b_{ij}$ and using the calculations
we have just made, we can write this in the form
\[\E_{x_u,x_v}\omega^{\beta_{uv}(x_u,x_v)+\psi_u(x_u)+
\psi_v(x_v)-b}.\]
If $u=v$, then the expectation is just over $x_u$ and the 
exponent has the form $q(x_u)+w^Tu-b$ for
some quadratic form $q$ of rank at least $r$. Therefore,
by Lemma \ref{gauss}, the expectation is at most $p^{-r/2}$. If 
$u\ne v$ (and therefore every $\b_{uu}$ is zero) 
then for each $x_v$ the exponent is linear in $u$. This means
that either the expectation over $x_u$ is zero or the function
$\beta_{uv}(x_u,x_v)+\psi_u(x_u)$ is constant. If the latter
is true when $x_v=y$ and when $x_v=z$, then $\beta_{uv}(x_u,y-z)$
is also constant, and therefore identically zero. Since 
$\beta_{uv}$ has rank at least $r$, $y-z$ must lie in a 
subspace of codimension at least $r$. Therefore, the set of
$x_v$ such that $\beta_{uv}(x_u,x_v)+\psi_u(x_u)$ is constant
is an affine subspace of $\F_p^n$ of codimension at least $r$,
which implies that the probability (for a randomly chosen $x_v$)
that the expectation (over $x_u$) is non-zero is at most $p^{-r}$. 
When the expectation is non-zero, it has modulus 1. 

In either case, we find that, for any non-zero $\lambda\in\Lambda$, 
the expectation over $\x$ is at most $p^{-r/2}$, and this completes
the proof of the lemma.
\end{proof}

We now want to take into account $\Gamma_1$ as well as $\Gamma_2$.
This turns out to be a short deduction from the previous result.
First let us do a simple piece of linear algebra.

\begin{lemma}\label{linfactor}
Let $\Lsys=(L_1,\dots,L_m)$ be a collection of
linear forms in $d$ variables, and suppose that the linear span
of $L_1,\dots,L_m$ has dimension $d'$. Let 
$\Gamma_1:\F_p^n\rightarrow\F_p^{d_1}$ be a surjective linear map
and let $\phi:(\F_p^n)^d\rightarrow(\F_p^{d_1})^m$ be defined by 
the formula
\[\phi:\x\mapsto(\Gamma_1(L_1(\x)),\dots,\Gamma_1(L_m(\x))).\]
Then the image of $\phi$ is the subspace $Z$ of $(\F_p^{d_1})^m$
that consists of all sequences $(a_1,\dots,a_m)$ such that
$\sum_i\mu_ia_i=0$ whenever $\sum_i\mu_iL_i=0$. The dimension
of $Z$ is $d'd_1$.
\end{lemma}

\begin{proof}
Since the $m$ forms $L_i$ span a space of dimension $d'$, the set 
of sequences $\mu=(\mu_1,\dots,\mu_m)$ such that $\sum_i\mu_iL_i=0$
is a linear subspace $W$ of $\F_p^m$ of dimension $m-d'$. Therefore,
the condition that $\sum_i\mu_ia_i=0$ for every sequence $\mu\in W$
restricts $(a_1,\dots,a_m)$ to a subspace of $(\F_p^{d_1})^m$
of codimension $d_1(m-d')$. (An easy way to see this is to write
$a_i=(a_{i1},\dots,a_{id_1})$ and note that for each $j$ the sequence
$(a_{1j},\dots,a_{mj})$ is restricted to a subspace of codimension
$m-d'$.) Therefore, the dimension of $Z$ is $d'd_1$, as claimed.

Now let us show that $Z$ is the image of $\phi$. Since $\phi$ is 
linear, $Z$ certainly contains the image of $\phi$, so it will be
enough to prove that the rank of $\phi$ is $d'd_1$. 

Abusing notation, let us write $\Gamma_1(\x)$ for the sequence
$(\Gamma_1x_1,\dots,\Gamma_1x_d)$, which belongs to $(\F_p^{d_1})^d$. 
Then $\phi(\x)$ can be rewritten
as $(L_1(\Gamma_1(\x)),\dots,L_m(\Gamma_1(\x)))$.
Since $\Gamma_1$ is a surjection, it is also a surjection when
considered as a map on $(\F_p^n)^d$. Therefore, the rank of $\phi$ 
is the rank of the map $\psi:(\F_p^{d_1})^d\rightarrow(\F_p^{d_1})^m$ 
defined by
\[\psi:\y\mapsto(L_1(\y),\dots,L_m(\y)).\]
Since the $L_i$ span a space of dimension $d'$, the nullity of
this map is $d_1(d-d')$, so its rank is $d_1d'$. Therefore,
the image of $\phi$ is indeed $Z$.
\end{proof}

\begin{lemma}\label{completefactor}
Let $\Lsys=(L_1,\dots,L_m)$ be a square-independent system of
linear forms in $d$ variables, and suppose that the linear span
of $L_1,\dots,L_m$ has dimension $d'$. Let 
$\Gamma_1:\F_p^n\rightarrow\F_p^{d_1}$ be a surjective linear map and 
let $\Gamma_2:\F_p^n\rightarrow\F_p^{d_2}$ be a quadratic map of 
rank at least $r$.
Let $a_1,\dots,a_m$ be elements of $\F_p^{d_1}$ and let 
$b_1,\dots,b_m$ be elements of $\F_p^{d_2}$, and let $\phi$ 
and $Z$ be as defined in the previous lemma. Then the probability, if 
$\x$ is chosen randomly from $(\F_p^n)^d$, that
$\Gamma_1(L_i(\x))=a_i$ and $\Gamma_2(L_i(\x))=b_i$ for
every $i\leq m$ is zero if $(a_1,\dots,a_m)\in Z$, and 
otherwise it differs from $p^{-d_1d'-d_2m}$ by at most $p^{d_1-d'd_1-r/2}$.
\end{lemma}

\begin{proof}
If $\a=(a_1,\dots,a_m)\notin Z$, then there exists $\mu\in\F_p^m$
such that $\sum_i\mu_ia_i\ne 0$ and $\sum_i\mu_iL_i(\x)=0$ for
every $\x$. Since $\Gamma_1$ is linear, it follows that there
is no $\x$ such that $\Gamma_1(L_i(\x))=a_i$ for every~$i$.

Otherwise, by Lemma \ref{linfactor}, $\a$ lies in the image of $\phi$, which
has rank $d'd_1$, so $\phi^{-1}(\{\a\})$ is an affine subspace
of $(\F_p^n)^d$ of codimension $d'd_1$. Therefore, the probability
that $\phi(\x)=\a$ is $p^{-d'd_1}$. Now let us use Lemma \ref{quadfactor} to
estimate the probability, conditional on this, that 
$\Gamma_2(L_i(\x))=b_i$ for every $i\leq m$.

In the proof of Lemma \ref{linfactor}, we observed that $\phi(\x)$
depends on $\Gamma_1(\x)$ only, so we shall estimate the required
probability, given the value of $\Gamma_1(\x)$. (Recall that this is
notation for $(\Gamma_1x_1,\dots,\Gamma_1x_d)$.) In order to specify
the set on which we are conditioning, let $V$ be the kernel of
$\Gamma_1$ (considered as a map defined on $\F_p^n$), and given a
sequence $(w_1,\dots,w_d)\in(\F_p^n)^d$, let us estimate the required
probability, given that $x_u\in V+w_u$ for every $u$.

Let us write $x_u=y_u+w_u$. Thus, we are estimating the probability
that $\Gamma_2(L_i(\y+\w))=b_i$ for every $i\leq m$. But for each
$i$ we can write $\Gamma_2(L_i(\y+\w))$ as 
$\Gamma_2(L_i(\y))+\phi_i(\y)+b_i'$ for some linear function 
$\phi_i:V^d\rightarrow\F_p^{d_2}$ and some vector $b_i'\in\F_p^{d_2}$.

Because $\Gamma_2$ has rank at least $r$ and the codimension of $V$ in
$\F_p^n$ is $d_1$, Lemma \ref{rankrestr} implies that the rank of the
restriction of $\Gamma_2$ to $V$ is at least $r-2d_1$. Therefore, by
Lemma \ref{quadfactor}, the probability that
$\Gamma_2(L_i(\y))=-\phi_i(\y)+b_i-b_i'$ for every $i$ differs from
$p^{-md_2}$ by at most $p^{d_1-r/2}$.

Since this is true for all choices of $\w$, we have the same estimate
if we condition on the event that $\phi(\x)=\a$ for some fixed $\a\in Z$. 
Therefore, the probability that $\Gamma_1(L_i(\x))=a_i$ and 
$\Gamma_2(L_i(\x))=b_i$ for every $i$ differs from $p^{-d'd_1-md_2}$
by at most $p^{d_1-d'd_1-r/2}$, as claimed.
\end{proof}

Next, we observe that Lemma \ref{completefactor} implies that all the atoms
of $\btwo$ have approximately the same size.

\begin{corollary}\label{atomsize}
Let $\Gamma_1$ and $\Gamma_2$ be as above and let $x$ be a randomly
chosen element of $\F_p^n$. Then for every $a\in\F_p^{d_1}$
and every $b\in\F_p^{d_2}$, the probability that $\Gamma_1(x)=a$
and $\Gamma_2(x)=b$ differs from $p^{-d_1-d_2}$ by at most $p^{-r/2}$.
\end{corollary}

\begin{proof}
Let us apply Lemma \ref{completefactor} in the case where $\Lsys$ consists of
the single one-variable linear form $L(x)=x$. This has linear rank 1
and is square-independent, so when we apply the lemma we have 
$d'=m=1$. If we let $a_1=a$ and $b_1=b$, then the conclusion of the 
lemma tells us precisely what is claimed.
\end{proof}

The next two lemmas are simple technical facts about projections
on to linear factors. The first one tells us that if $g$ is any 
function that is uniform and constant on the atoms of a linear 
factor, then it has small $L_2$-norm. The second tells us that
projecting on to a linear factor decreases the $U^2$-norm.

\begin{lemma}\label{younginverse}
Let $G$ be a function from $\F_p^{d_1}$ to $[-1,1]$, let 
$\Gamma_1:\F_p^n\rightarrow\F_p^{d_1}$ be a surjective linear 
map and let $g=G\circ\Gamma_1$. Then $\|g\|_2^4\leq p^{d_1}\|g\|_{U^2}^4$.
\end{lemma}

\begin{proof}
Since $\Gamma_1$ takes each value in $\F_p^{d_1}$ the same
number of times, $\|g\|_{U^2}=\|G\|_{U^2}$. But
\[\|G\|_{U^2}^4=\E_a(\E_bG(b)G(b+a))^2\geq p^{-d_1}(\E_bG(b)^2)^2
=p^{-d_1}\|G\|_2^4,\]
which proves the result, since $\|g\|_2=\|G\|_2$ as well.
\end{proof}

\begin{lemma}\label{u2pyth}
Let $f$ be a function from $\F_p^n$ to $\R$,
let $\bone$ be a linear factor and let $g=\E(f|\bone)$. Then
$\|g\|_{U^2}\leq\|f\|_{U^2}$.
\end{lemma}

\begin{proof}
On every atom of $\bone$, $g$ is constant and $f-g$ averages zero. 
Let $\Gamma_1$ be
the linear map that defines $\bone$ and, as we did earlier,
for each $a\in\F_p^{d_1}$ let $V_a$ stand for $\Gamma_1^{-1}(\{a\})$. 
Then
\[\|f\|_{U^2}^4=\E_{a_1+a_2=a_3+a_4}\E_{x_1+x_2=x_3+x_4\atop
\Gamma_1(x_i)=a_i}f(x_1)f(x_2)f(x_3)f(x_4)\ .\]
Let us fix a choice of $a_1+a_2=a_3+a_4$ and consider the inner
expectation. Setting $g'=f-g$, this has the form
\[\E_{x_1+x_2=x_3+x_4\atop\Gamma_1(x_i)=a_i}
(\lambda_1+g'(x_1))(\lambda_2+g'(x_2))
(\lambda_3+g'(x_3))(\lambda_4+g'(x_4))\]
This splits into sixteen parts. Each part that involves at least
one $\lambda_i$ and at least one $g'(x_i)$ is zero, because any
three of the $x_i$s can vary independently and $g'$ averages zero
on every atom of $\bone$. This means that the expectation is
\[\lambda_1\lambda_2\lambda_3\lambda_4+
\E_{x_1+x_2=x_3+x_4\atop
\Gamma_1(x_i)=a_i}g'(x_1)g'(x_2)g'(x_3)g'(x_4)\ .\]
If we now take expectations over $a_1+a_2=a_3+a_4$ we find that
$\|f\|_{U^2}^4=\|g\|_{U^2}^4+\|f-g\|_{U^2}^4$. Notice that this
is a general result about how the $U^2$-norm of a function is
related to the $U^2$-norm of a projection on to a linear factor.
\end{proof}

Now we are ready to estimate the product we are interested
in, for functions that are constant on the atoms of $\btwo$. 

\begin{lemma}\label{bound1}
Let $\Gamma_1:\F_p^n\rightarrow\F_p^{d_1}$ be a linear function
and $\Gamma_2:\F_p^n\rightarrow\F_p^{d_2}$ be a quadratic function.
Let $(\bone,\btwo)$ be the corresponding quadratic factor and 
suppose that this has rank at least $r$. Let $c>0$ and let 
$f:\F_p^n\rightarrow[-1,1]$ be a function with $\|f\|_{U^2}\leq c$
and let $f_1=\E(f|\btwo)$. Let $\Lsys=(L_1,\dots,L_m)$ be a 
square-independent system of linear forms. Then
\[\E_{\x\in(\F_p^n)^d}\prod_{i=1}^mf_1(L_i(\x))\leq
4^mcp^{d_1/4}+2^{m+1}p^{m(d_1+d_2)-r/2}.\]
\end{lemma}

\begin{proof}
Let $g=\E(f_1|\bone)$ and let $h=f_1-g$. Then 
$\|g\|_1\leq\|g\|_2\leq p^{d_1/4}\|g\|_{U^2}$,
by Cauchy-Schwarz and Lemma \ref{younginverse}. By Lemma \ref{u2pyth},
$\|g\|_{U^2}\leq\|f\|_{U^2}$, which is
at most $c$, by hypothesis. Therefore, $\|g\|_1\leq cp^{d_1/4}$.

Since $f_1=g+h$, we can split the product up into a sum of $2^m$
products, in each of which we replace $f_1(L_i(\x))$ by either
$g(L_i(\x))$ or $h(L_i(\x))$. Since $\|g\|_1\leq cp^{d_1/4}$
and $\|h\|_\infty\leq 2$, any product that involves at least one
$g$ has average at most $2^mcp^{d_1/4}$. It remains to estimate
\[\E_{\x\in(\F_p^n)^d}\prod_{i=1}^mh(L_i(\x)).\]

Let $Z$ be as defined in Lemma \ref{linfactor}, and for each
$\a=(a_1,\dots,a_m)$ and $\b=(b_1,\dots,b_m)$, let $P(\a,\b)$ be the
probability that $\Gamma_1(L_i(\x))=a_i$ and $\Gamma_2(L_i(\x))=b_i$
for every $i$.  By Lemma \ref{completefactor}, we can set
$P(\a,\b)=p^{-d'd_1-md_2}+\epsilon(\a,\b)$, with
$|\epsilon(\a,\b)|\leq p^{d_1-d'd_1-r/2}$.

Now let $H$ be defined by the formula 
$h(x)=H(\Gamma_1x,\Gamma_2x)$. Because $h$ is constant on the 
atoms of $\btwo$, $H$ is well-defined on the set of all elements
of $\F_p^{d_1}\times\F_p^{d_2}$ of the form $(\Gamma_1x,\Gamma_2x)$.
Since $h$ takes values in $[-2,2]$, so does $H$.

Next, we show that $\E_bH(a,b)$ is small for any fixed
$a\in\F_p^{d_1}$, using the facts that $h$ averages $0$ on every cell
of $\bone$ and that it is constant on the cells of $\btwo$. Let us
fix an $a$ and write $P(b)$ for the probability that $\Gamma_2(x)=b$ 
given that $\Gamma_1(x)=a$---that is, for the density of $V_a\cap W_b$
inside $V_a$. Then
\[0=\E_{x\in V_a}h(x)=\E_{x\in V_a}H(\Gamma_1x,\Gamma_2x)=\sum_b P(b)H(a,b).\]
By Corollary \ref{atomsize}, we can write $P(b)=p^{-d_2}+\epsilon(b)$, with
$|\epsilon(b)|\leq p^{d_1-r/2}$ for every $b$. Therefore, the right-hand
side differs from $\E_bH(a,b)$ by at most $2p^{d_1+d_2-r/2}$, which
implies that $|\E_bH(a,b)|\leq 2p^{d_1+d_2-r/2}$.

Now
\[\E_\x\prod_{i=1}^mh(L_i(\x))=\E_\x\prod_{i=1}^mH(\Gamma_1(L_i(\x)),
\Gamma_2(L_i(\x)))=\sum_{\a\in Z}\sum_\b P(\a,\b)\prod_{i=1}^mH(a_i,b_i).\]
Let us split up this sum as
\[p^{-d'd_1-md_2}\sum_{\a\in Z}\sum_\b\prod_{i=1}^mH(a_i,b_i)
+\sum_{\a\in Z}\sum_\b \epsilon(\a,\b)\prod_{i=1}^mH(a_i,b_i).\]

The first term equals $\E_{a\in Z}\prod_{i=1}^m(\E_bH(a_i,b))$,
which is at most $(2p^{d_1+d_2-r/2})^m$. The second is at most
$p^{(d'd_1+md_2)}2^mp^{d_1-d'd_1-r/2}=2^mp^{d_1+md_2-r/2}$. 
Therefore, the whole sum is at most $2^{m+1}p^{m(d_1+d_2)-r/2}$.
Together with our estimate for the terms that involved $g$,
this proves the lemma. 
\end{proof}

We have almost finished the proof of our main result.

\begin{theorem}\label{bound2}
For every $\epsilon>0$ there exists $c>0$ 
with the following property. Let $f:\F_p^n\rightarrow[-1,1]$
be a $c$-uniform function. Let $\Lsys=(L_1,\dots,L_m)$ be a
square-independent system of linear forms in $d$ variables, 
with Cauchy-Schwarz complexity at most 2. Then 
\[\E_{\x\in(\F_p^n)^d}\prod_{i=1}^m f(L_i(\x))\leq\epsilon.\]
\end{theorem}

\begin{proof}
Let $\delta>0$ be a constant to be chosen later. Let $C$ be 
such that $2^{m+1}p^{-C/2}\leq\epsilon/3$ and let
$r$ be the function $d\mapsto 2md+C$. Then according to the 
structure theorem (Theorem \ref{qd2}) there exists $d_0$, depending
on $r$ and $\delta$ only, and a quadratic factor $(\bone,\btwo)$
of rank at least $2m(d_1+d_2)+C$ and complexity $(d_1,d_2)$, with
$d_1$ and $d_2$ both at most $d_0$, such that we can write
$f=f_1+f_2+f_3$, with $f_1=\E(f|\btwo)$, $\|f_2\|_2\leq\delta$
and $\|f_3\|_{U^3}\leq\delta$.

Let us show that the sum does not change much if we replace
$f(L_m(\x))$ by $f_1(L_m(\x))$. The difference is what we get
if we replace $f(L_m(\x))$ by $f_2(L_m(\x))+f_3(L_m(\x))$.
Now $\|f_2\|_1\leq\|f_2\|_2$ and $\|f\|_\infty\leq 1$, so 
the contribution from the $f_2$ part is at most $\delta$.
As for the $f_3$ part, since $\|f_3\|_{U^3}\leq\delta$
and $\|f\|_\infty\leq 1$, Theorem \ref{gvn} tells us that the
contribution is at most $\delta$. Therefore, the
total difference is at most $\delta+\delta\leq 2\delta$.

Now let us replace $f$ by $f_1$ in the penultimate bracket. The
same argument works, since $\|f_1\|_\infty\leq 1$. Indeed, we
can carry on with this process, replacing every single $f$ by
$f_1$, and the difference we make will be at most $2m\delta$.

We are left needing to show that the product with every $f$
replaced by $f_1$ is small. This is what Lemma \ref{bound1} tells us.
It gives us an upper bound of $4^mcp^{d_1/4}+2^{m+1}p^{m(d_1+d_2)-r/2}$,
where for $r$ we can take $2m(d_1+d_2)+C$. Therefore, the upper
bound is $4^mcp^{d_0/4}+2^{m+1}p^{-C/2}$, which, by our choice
of $C$, is at most $4^mcp^{d_0/4}+\epsilon/3$. 

To finish, let $\delta=\epsilon/6m$. This determines the value of
$d_0$ and we can then set $c$ to be $4^{-m}p^{-d_0/4}\epsilon/3$,
which will be a function of $\eps$ only.
\end{proof}

Because of our use of Theorem \ref{qd2}, the bounds in the above
result and in the corollary that we are about to draw from it are
both very weak. However, we have been explicit about all the
bounds apart from $d_0$, partly in order to make it clear how
the parameters depend on each other and partly to demonstrate
that our weak bound derives just from the weakness of $d_0$ in 
the structure theorem.

\begin{corollary}\label{bound3}
For every $\epsilon>0$ there exists $c>0$ 
with the following property. Let $A$ be a $c$-uniform subset
of $\F_p^n$ of density $\alpha$. Let $\Lsys=(L_1,\dots,L_m)$ be a
square-independent system of linear forms in $d$ variables, 
with Cauchy-Schwarz complexity at most 2. Let $\x=(x_1,\dots,x_d)$ 
be a random element of $(\F_p^n)^d$. Then the probability
that $L_i(\x)\in A$ for every $i$ differs from $\alpha^m$ by
at most $\epsilon$.
\end{corollary}

\begin{proof}
We shall choose as our $c$ the $c$ that is given by the previous
theorem when $\epsilon$ is replaced by $\epsilon/2^m$.
Our assumption is then that we can write $A=\alpha+f$ for a 
$c$-uniform function $f$. The probability we are interested 
in is 
\[\E_{\x\in(\F_p^n)^d}\prod_{i=1}^m A(L_i(\x)),\]
which we can split into $2^m$ parts, obtained
by replacing each occurrence of $A$ either by $\alpha$
or by $f$. 

For each part that involves at least one occurrence of $f$, we have a
power of $\alpha$ multiplied by a product over some subsystem of
$\Lsys$. This subsystem will also be square-independent and have
CS-complexity at most 2. Moreover, the number of linear forms will
have decreased. Therefore, the previous theorem and our choice of $c$
tell us that the contribution it makes is at most
$\epsilon/2^m$. Therefore, the contribution from all such parts is at
most $\epsilon$. The only remaining part is the one where every
$A(L_i(\x))$ has been replaced by $\alpha$, and that gives us the main
term $\alpha^m$.
\end{proof}

\begin{section}{Concluding remarks}

First, we remark that Corollary \ref{bound3} allows us to deduce
rather straightforwardly a Szemer\'edi-type theorem for
square-independent systems of CS-complexity 2 which have the
additional property that they are \emph{translation-invariant}. That
is, one can show that any sufficiently dense subset of $\F_p^n$
contains a configuration of the given type.

Without the result of the preceding section, establishing that any
sufficiently dense subset contains a solution to systems of this type
would require a quadratic argument of the form used by Green and Tao
to prove Szemer\'edi's Theorem for progressions of length $4$ in
finite fields \cite{greentao:r4boundsI}. This would involve obtaining
density increases on quadratic subvarieties of $\F_p^n$, which then
need to be linearized in a carefully controlled manner. Although it is
certainly possible to adapt their argument in this way, for purely
qualitative purposes it is much simpler to use the result that
configurations of this type are governed by the $U^2$-norm, which
allows one to produce a density increase on an affine subspace. The
resulting argument is almost identical to the well-known argument for
$3$-term progressions \cite{meshulam:3APs}.  Translation invariance is
needed because the subspace on which we find a density increment may
be an affine and not a strictly linear one. (It is not hard to show
that the result is false if the system is not translation invariant.)
Recall that two examples of square-independent translation invariant
systems of complexity 2 are the systems $(x, x+y, x+z, x+y+z,
x+y-z, x+z-y)$ and $(x,x+a,x+b,x+c,x+a+b,x+a+c,x+b+c)$.

The second of these examples shows that our main result implies the 
following useful ``Pythagorean theorem,''  which generalizes the much
more straightforward fact that if $a$ is a constant and $f$ averages
zero, then $\|a+f\|_{U^2}^4=a^4+\|f\|_{U^2}^4$.

\begin{theorem}\label{pythagoras}
For every $\epsilon>0$ there exists $c>0$ such that if $f$ is a
$c$-uniform function from $\F_p^n$ to $[-1,1]$, $a\in[-1,1]$
is a constant, and $g(x)=a+f(x)$ for every $x\in\F_p^n$, then 
$\bigl|\|g\|_{U^3}^8-(a^8+\|f\|_{U^3}^8)\bigr|\leq\epsilon$.
\end{theorem}

We briefly sketch the proof: expanding out the definition of
$\|a+f\|_{U^3}^8$ one obtains a sum of $2^8$ terms, one of which
gives $a^8$ (if you choose $a$ from every bracket) and one of
which gives $\|f\|_{U^3}^8$ (if you choose $f$ from every bracket).
All the remaining terms are constant multiples of expectations
of $f$ over linear configurations that are square-independent
and therefore, by our main result, small.
\medskip

There are several ways in which the results of Section 3 might be
generalized. An obvious one is to prove comparable results for the
group $\Z_N$. As we mentioned earlier, we have a different proof
for $\F_p^n$ and this can be transferred to $\Z_N$ by ``semi-standard''
methods. (That is, the general approach is clear, but the details 
can be complicated and sometimes require more than merely technical
thought.) The alternative proof for $\F_p^n$ gives a doubly exponential
bound for the main result rather than the tower-type bound obtained
here.

Possibly even more obvious is to try to extend the main result of this
paper to a proof of Conjecture \ref{mainconj}. This involves a
generalization in two directions: to systems of CS-complexity greater
than 2, and to systems with true complexity greater than 1.  All
further cases will require polynomial Fourier analysis for a degree
that is greater than 2: the simplest is likely to be to show that a
square-independent system with CS-complexity 3 has true complexity
1. In this case, we would use a decomposition into a structured part
(a projection onto a cubic factor) and a uniform part (which would be
small in $U^4$ and therefore negligible) and then, as before,
concentrate on the structured part. Square-independence (which implies
cube-indepence) would ensure that we could reduce to the linear part
of the factor as before.

This state of affairs leaves us very confident that Conjecture
\ref{mainconj} is true. Although cubic and higher-degree Fourier
analysis have not yet been worked out, they do at least exist in local
form for $\Z_N$: they were developed in \cite{gowers:SzT} to prove the
general case of Szemer\'edi's theorem. It is therefore almost certain
that global forms will eventually become available, both for $\Z_N$
and for $\F_p^n$. And then, given a
statement analogous to Theorem \ref{qd2}, it is easy to see how to
generalize the main steps of our proof. In particular, the Gauss-sum
estimates on which we depend so heavily have higher-degree
generalizations.

A completely different direction in which one might consider 
generalizing the above results is to hypergraphs. For example, 
very similar proofs to those of Theorems \ref{vn} and \ref{gvn} 
can be used to prove so-called ``counting lemmas'' for quasirandom
hypergraphs---lemmas that assume that a certain norm is small and
deduce that the hypergraph contains approximately the expected
number of small configurations of a given kind.

One can now ask whether, as with sets, weaker quasirandomness
assumptions about a hypergraph suffice to guarantee the right number
of certain configurations, and if so, which ones. It turns out to be
possible to give a complete answer to a fairly natural formulation of
this question. Unfortunately, however, the proof is rather too easy to
be interesting, so here we content ourselves with somewhat informal
statements of results concerning the special case of 3-uniform
hypergraphs.  The proofs we leave as exercises for any reader who
might be interested.

Recall that if $X$, $Y$ and $Z$ are finite sets and 
$f:X\times Y\times Z\rightarrow\R$, then the \emph{octahedral norm}
of $f$ is the eighth root of
\[\E_{x(0),x(1)\in X}\E_{y(0),y(1)\in Y}\E_{z(0),z(1)\in Z}
\prod_{\epsilon\in\{0,1\}^3}f(x(\epsilon_1),y(\epsilon_2),z(\epsilon_3)).\]
It is easy to verify that if $X=Y=Z=G$ for some Abelian group $G$
and $f(x,y,z)=g(x+y+z)$ for some function $g$, then the octahedral
norm of $f$ is the same as the $U^3$-norm of $g$. Therefore, it is
natural to consider the octahedral norm of functions defined on 
$X\times Y\times Z$ as the correct analogue of the $U^3$-norm of
functions defined on Abelian groups.

An important fact about the octahedral norm is that $f$ has small
octahedral norm if and only if it has a small correlation with any
function of the form $u(x,y)v(y,z)w(x,z)$. Another important fact, the
so-called ``counting lemma'' for quasirandom hypergraphs, states the
following. Let $X$ be a finite set and let $H$ be a 3-uniform
hypergraph with vertex set $X$ and density $\alpha$. Suppose that $H$
is quasirandom in the sense that the function $H(x,y,z)-\alpha$ has
small octahedral norm (where $H(x,y,z)=1$ if $\{x,y,z\}\in H$ and 0
otherwise). Then $H$ has about the expected number of copies of any
fixed small hypergraph. For instance, if you choose $x$, $y$, $z$ and
$w$ randomly from $X$, then the probability that all of
$\{x,y,z\},\{x,y,w\},\{x,z,w\}$ and $\{y,z,w\}$ belong to $H$ is
approximately $\alpha^4$.

Now let us suppose that $g$ is uniform but not necessarily quadratically
uniform, and that we again define $f(x,y,z)$ to be $g(x+y+z)$. What
can we say about $f$? It is not necessarily the case that $f$ has
small octahedral norm, or that it has low correlation with functions
of the form $u(x,y)v(y,z)w(x,z)$. However, it is not hard to show that
it has low correlation with any function of the form $a(x)b(y)c(z)$,
a property that was referred to as \emph{vertex uniformity} in 
\cite{gowers:HRL3}. 

One might therefore ask whether vertex uniformity was sufficient to
guarantee the right number of copies of some small hypergraphs. 
However, well-known and easy examples shows that it does so only for 
hypergraphs such that no pair $\{x,y\}$ is contained in more than 
one hyperedge. For instance, let $u$ be a random symmetric function 
from $X^2$ to $\{-1,1\}$ and let $H(x,y,z)=(3+u(x,y)+u(y,z)+u(x,z))/6$.
Then $H$ is vertex uniform and has density $1/2$, but it is a simple 
exercise to show that $\E_{x,y,z,w}H(x,y,z)H(x,y,w)$ is about $5/18$ 
instead of the expected $1/4$.

However, this is perhaps not the right question to be asking. If $g$
is uniform, then $f$ has a stronger property than just vertex
uniformity: one can prove that it does not correlate with any function
of the form $u(x,y)w(x,z)$, $u(x,y)v(y,z)$ or $v(y,z)w(x,z)$. If we
take \emph{this} as our definition of ``weak quasirandomness'' for
functions (and call the hypergraph $H$ weakly quasirandom if the
function $H-\alpha$ is), then which hypergraphs appear with the right
frequency (or with ``frequency zero'' if we are talking about
functions rather than sets)? The answer turns out to be that a sum
over copies of a small hypergraph $H'$ will have the ``right'' value
if and only if there is a pair $\{x,y\}$ that belongs to exactly one
hyperedge $\{x,y,z\}$ of $H'$. The proof in the ``if'' direction
is an easy exercise. In particular, it does not involve any
interesting results about decomposing hypergraphs, which suggests that
the main result of this paper is, in a certain sense, truly
arithmetical.

As for the ``only if'' direction, here is a quick indication of how to
produce an example (in the complex case, for simplicity). Suppose that
no pair $\{x,y\}$ belongs to more than $m$ hyperedges in $H'$. For
each $k$ between 2 and $m$ let $f_k:X^2\rightarrow\C$ be a function 
whose values are randomly chosen $k$th roots of unity. Then let
$f(x,y,z)$ be the sum of all functions of the form $u(x,y)v(y,z)w(x,z)$,
where each of $u$, $v$ and $w$ is some $f_k$ with $2\leq k\leq m$.
When one expands out the relevant sum for this function $f$, one
finds that most terms cancel, but there will be some that don't
and they will all make a positive contribution. To find such a
term, the rough idea is to choose for each face $F$ of $H'$ a triple 
of functions $(f_{k_1},f_{k_2},f_{k_3})$, where $k_1$, $k_2$ and
$k_3$ are the number of faces of $H'$ that include each of the
three edges that make up the face $F$. For this term, each time
a $k$th root of unity appears in the product, it is raised to the
power $k$, so the term is large.

Finally, let us say just a little bit more about the result of Leibman
mentioned at the beginning of the paper. The question in ergodic
theory which is analogous to the one we were studying in Section 3
concerns so-called ``characteristic factors'' for ergodic averages of
the form
\[\lim_{N \tends \infty} \frac{1}{N^d} \sum_{n_1, ..., n_d=1}^{N} \int
T^{L_1(\n)} f_1 (x) \;T^{L_2(\n)} f_2 (x)\; \dots \;T^{L_m(\n)} f_m
(x) d\mu(x),\] 
where $T$ is a measure-preserving map on a probability measure space
$(X, \B,\mu)$ and the functions $f_i$ belong to $L^{\infty}(\mu)$. A
characteristic factor is a system onto which one can project without
losing any quantitative information about the average under
consideration. The aim is to find characteristic factors which possess
enough structure to allow one to establish convergence of the above
average in a rather explicit way. For example, it was shown by Host
and Kra \cite{hostkra:annals} and Ziegler \cite{ziegler:universalfactors} 
independently that when the linear forms $L_1, ..., L_m$ describe an arithmetic
progression of length $m$, then there exists a characteristic factor
for the average which is isomorphic to an inverse limit of a
sequence of $(m-2)$-step nilsystems. For $m=4$, these very structured
objects are closely related to the quadratic factor we are using in
this paper, on which computations can be performed rather
straightforwardly. After these remarks it should not be surprising
that there is a notion of degree associated with a characteristic
factor. What we have called the true complexity of a linear system is
closely analogous to the degree of the minimal characteristic factor.

Leibman \cite{leibman:new} characterizes the degree of the minimal
characteristic factor for general linear as well as certain polynomial
systems. As his definition of complexity in the ergodic context is
highly technical, we shall simply illustrate the analogy with our
result by quoting two examples from Section 6 of his paper: In our
terminology, both of the systems given by $(x+n+m, x+2n+4m, x+3n+9m,
x+4n+16m, x+5n+25m, x+6n+36m)$ and the ever so slightly different
$(x+n+m, x+2n+4m, x+3n+9m, x+4n+16m, x+5n+25m, x+6n+37m)$ have
CS-complexity 2. However, the second one has true complexity $1$ since
its squares are independent, or, as Leibman puts it, because the six
vectors $(1,1,1,1,1,1)$, $(1,c_1, c_2, \dots, c_5)$, $(1,d_1, d_2, \dots,
d_5)$, $(1,c_1^2, c_2^2, \dots, c_5^2)$, $(1,d_1^2, d_2^2, \dots, d_5^2)$ and
$(1, c_1 d_1, c_2 d_2, \dots, c_5 d_5)$ span $\R^6$. (Here $c_i, d_i$ are
the cofficients of $n,m$, respectively, in the linear form $i+1$. Note
that the special form of the ergodic average forces one to consider
translation-invariant systems only, which leads to a formulation of
square-independence that is particular to systems where one variable
has coefficient 1 in all linear forms.) 

In his proof of Szemer\'edi's Theorem, Furstenberg
\cite{furstenberg:SzT} developed an important tool known as the
\emph{correspondence principle} which allowed him to deduce
Szemer\'edi's combinatorial statement from the recurrence properties
of a certain dynamical system. Our result in the $\Z_N$ case does not
appear to follow from Leibman's result by a standard application of
the correspondence principle.  For an excellent introduction to
ergodic theory and its connections with additive combinatorics, we
refer the interested reader to \cite{kra:montreal}.

\end{section}

\end{document}